\documentclass[a4paper,12pt]{article}
\usepackage{times, url}
\usepackage{amsmath,amsthm} 
\usepackage[lite]{amsrefs}
\usepackage{chngcntr}
\usepackage[english]{babel}
\usepackage{enumitem}
\usepackage{hyperref} 
\usepackage{amsthm}
\usepackage{amssymb}
\usepackage[all,cmtip]{xy}
\usepackage{amsmath,amssymb,amsthm,amsfonts}

\newtheorem{theorem}{Theorem}[section]
\newtheorem{lemma}[theorem]{Lemma}

\newtheorem{corollary}[theorem]{Corollary}
\theoremstyle{definition}

\newtheorem*{case 1*}{Case 1}
\newtheorem*{case 2*}{Case 2}
\begin{document}

\setcounter{page}{1}

\begin{center}
{\LARGE \bf  An extension of Wilson's Theorem.}
\vspace{8mm}

{\large \bf Gaitanas Konstantinos}
\vspace{3mm}

Department of  Applied Mathematical and Physical Sciences\\
 National Technical University of Athens \\ 
Heroon Polytechneiou Str., Zografou Campus, 15780 Athens, Greece \\ 
e-mail: \url{kostasgaitanas@gmail.com}
\vspace{2mm}

\end{center}
\vspace{10mm}

\noindent
{\bf Abstract:} Let $\mathcal{N}[k]$ be the multiset containing the $\binom{n-1}{k}$ products of $k$-subsets of $\{1,\ldots, n-1\}$. We show that if $n\geq (2c+3)^2$, then 
\begin{gather*}\left((-1)^c+\sum_{M\in \mathcal{N}[n-1-c]}M\right)\cdot(c+1)\equiv 0\pmod{n},\end{gather*}
if and only if $n=(c+1)p$, where $p$ is prime. This provides a combinatorial extension of Wilson's Theorem, which is the special case where $c=0$.
\vspace{10mm}

\section{Introduction} 
The following theorem, known as \emph{Wilson's Theorem} provides probably the oldest and most famous non-trivial primality criterion:
\begin{theorem}\label{theorem} A positive integer $n>1$ is prime if and only if 
\begin{equation}\label{eq1}(n-1)!\equiv-1\pmod{n}.\end{equation}
\end{theorem}

\begin{proof}
A straightforward computation shows that the result holds true for $n=2$, so we may assume that $n\geq 3$.\par
If $n$ is prime, then $\mathbb{Z}_n$ is a field. This implies that the numbers $1, 2, \ldots, n-1$ with the exception of $\pm 1$, can be arranged in disjoint pairs $(x, x')$, such that $x\cdot x'\equiv 1\pmod{n}$. Thus, $(n-1)!\equiv -1\pmod{n}$, which shows that \eqref{eq1} holds true if $n$ is prime.\par 
If $n$ is composite, there is an integer $d>1$ such that $d\mid n$. Thus, $\gcd\left((n-1)!, n\right)\geq d$, which implies that $(n-1)!\equiv -1\pmod{n}$ is impossible. Thus, \eqref{eq1} does not hold true for composite numbers. This completes the proof.
\end{proof}

It is not easy to determine who was the first to give a valid proof since there is evidence that the result was known before the Middle Ages, but it was J. L. Lagrange who provided the first \emph{published} proof. For a historical account of Wilson's Theorem, we refer the reader to the well-known book \emph{History of the Theory of Numbers} \cite{Dickson} by Leonard Eugene Dickson.\par

Let $\mathcal{N}=\{1,\ldots, n-1\}$ and $\mathcal{N}[k]$ be the multiset containing the $\binom{n-1}{k}$ products of $k$-subsets of $\mathcal{N}$, that is
\begin{gather*}\mathcal{N}[k]=\{M: M=\prod_{m_{j_1},\ldots,m_{j_k}\in \mathcal{N}} m_{j_1}\cdots m_{j_k}\}.\end{gather*}
Since $\mathcal{N}[n-1]=\{(n-1)!\}$, Wilson's Theorem evidently states that $n>1$ is prime, if and only if 
\begin{gather*}\sum_{M\in \mathcal{N}[n-1]}M\equiv-1\pmod{n}.\end{gather*}
We would like to investigate whether some similar result holds true if we consider other values of $k$ less than $n-1$, namely when $k=n-1-c$. In particular, the primary objective of this paper is to extend Wilson's Theorem through the following question:
\begin{center} \emph{Let $c\geq 1$. What is the behavior of $\displaystyle\sum_{M\in \mathcal{N}[n-1-c]}M\pmod{n}$?}\end{center}
Despite the popularity of Wilson's Theorem, it seems that there has been no effort to extend it in this direction.
\section{Notation and preliminaries}
This section is a quick-reference guide to the notation and background information that will be assumed throughout this paper.\par
Throughout this paper, $p$ will always denote a prime number. The largest integer less than or equal to $x$ (usually called the \emph{floor} of $x$) will be denoted by $\lfloor x\rfloor$. Below we mention some useful properties of the floor function:
\begin{enumerate}
\item $\lfloor x\rfloor>x-1$.
\item If $k\nmid n$, then $\left\lfloor\frac{n-1}{k}\right\rfloor=\left\lfloor\frac{n}{k}\right\rfloor$.
\item\label{item3} $\left\lfloor\frac{n-1}{k}\right\rfloor\geq \frac{n}{k}-1$, for every $n, k\geq 2$.
\end{enumerate}
We briefly prove \eqref{item3}: If $k\mid n$, then $\frac{n}{k}$ is a positive integer and $k\nmid n-1$; thus, $\left\lfloor \frac{n-1}{k}\right\rfloor=\frac{n}{k}-1$. On the other hand, if $k\nmid n$, then (using the first two properties) $\left\lfloor \frac{n-1}{k}\right\rfloor=\left\lfloor \frac{n}{k}\right\rfloor> \frac{n}{k}-1$. In any case, property \eqref{item3} holds true.\par
Let $m\in\mathbb{N}$. We will write $p^a\mid \mid m$, if $a$ is the exponent of the largest power of $p$ that divides $m$. Legendre's formula states that $p^{\nu_p(m!)}\mid \mid m!$, where
\begin{gather*}\nu_p(m!)=\left\lfloor\frac{m}{p}\right\rfloor+\cdots +\left\lfloor\frac{m}{p^x}\right\rfloor, \quad p^x\leq m<p^{x+1}.\end{gather*}
A simple counting argument shows that 
\begin{equation}\label{eq2}\sum_{M\in \mathcal{N}[n-1-c]}M=\sum_{M\in \mathcal{N}[c]}\frac{(n-1)!}{M}.\end{equation}
If we denote by
\begin{gather*}\mathcal{N}_1[c]=\{M\in \mathcal{N}[c]: M=M'\cdot\prod_{\substack{p\mid m\\m<n}}m, \quad M'\in\mathbb{N}\}\end{gather*}
the multiset which contains all $M\in \mathcal{N}[c]$ which are divisible by all multiples of $p$ less than $n$, we may rewrite \eqref{eq2} in the following way:
\begin{gather*}\sum_{M\in \mathcal{N}[n-1-c]}M=\sum_{M\in \mathcal{N}_1[c]}\frac{(n-1)!}{M}+\sum_{M\not\in \mathcal{N}_1[c]}\frac{(n-1)!}{M}.\end{gather*}
Observe that if $M\not\in \mathcal{N}_1[c]$, then $\frac{(n-1)!}{M}\equiv 0\pmod{p}$; this is obvious, since there is at least one multiple of $p$ less than $n$ which divides $(n-1)!$ but not $M$. Hence,  
\begin{equation}\label{eq3}\sum_{M\in \mathcal{N}[n-1-c]}M\equiv\sum_{M\in \mathcal{N}_1[c]}\frac{(n-1)!}{M}\pmod{p}.\end{equation}
Note that this holds true, even if $\mathcal{N}_1[c]$ is empty.\footnote{We adopt the convention that the empty sum is zero.}\par
In closing, we will also make use of the following well-known congruence: 
\begin{equation}\label{eq4}(x-1)\cdots (\left(x-(p-1)\right)\equiv x^{p-1}-1\pmod{p},\quad  \footnote{From \eqref{eq2} we obtain that the constant term of the left-hand side (which is $(p-1)!$) is congruent to $-1\pmod{p}$. This provides an alternate proof of Wilson's Theorem.}\end{equation}
which means that the coefficients of the corresponding powers of $x$ are congruent modulo $p$. A proof of this result can be found in many mathematical contexts (such as \cite{Hardy}).\par
As for other prerequisites, the reader is expected to be familiar only with basic congruence rules and some standard tools from analysis.
\section{Some useful Lemmas}
We present two lemmas which will be useful for the proof of our main result, Theorem\eqref{main}. Both of them are obtained using only basic analysis.
\begin{lemma}\label{theorem1}
For every $n\geq 25$, the following inequality holds:
\begin{gather*}\frac{15n}{16}>\frac{(\sqrt{n}+1)\ln n}{2\ln 2}.\end{gather*} 
\begin{proof}
Let $f(x)=\frac{x}{\sqrt{x}+1}-\frac{8\ln x}{15\ln 2}, x\geq 25$. It's derivative is equal to 
\begin{gather*}f'(x)=\frac{\frac{\sqrt{x}}{2}+1}{(\sqrt{x}+1)^2}-\frac{8}{15\ln 2\cdot x}.\end{gather*}
But $8<15\ln 2$, which implies that $-\frac{8}{15\ln 2\cdot x}>-\frac{1}{x}$. Thus, 
\begin{gather*}f'(x)>\frac{\frac{\sqrt{x}}{2}+1}{(\sqrt{x}+1)^2}-\frac{1}{x}=\frac{\sqrt{x}(x-4)-2}{2x(\sqrt{x}+1)^2}.\end{gather*}
The denominator is always positive and the numerator is clearly positive since by assumption, $x\geq 25$. This implies that $f(x)$ is strictly increasing, thus, if $n\geq 25$, then $f(n)\geq f(25)>0.9$. This implies that $f(n)>0$, hence 
\begin{gather*}\frac{n}{\sqrt{n}+1}>\frac{8\ln n}{15\ln 2}\Leftrightarrow \frac{15n}{16}>\frac{(\sqrt{n}+1)\ln n}{2\ln 2}.\end{gather*} 
\end{proof}
\end{lemma}
\begin{lemma}\label{lemma2}Let $3\leq p\leq\sqrt{n}$. Then 
\begin{gather*}\left(\frac{n}{p}-1\right)\frac{\ln p}{\ln n}\geq \frac{\sqrt{n}-1}{2}.\end{gather*}
\begin{proof}
By assumption, $p\leq \sqrt{n}$, which implies that $\ln p\leq \frac{1}{2} \ln n$. With a little patience we can see that this is equivalent to 
\begin{gather*}\left(\frac{n}{p}-1\right)\frac{\ln p}{\ln n}\geq \frac{n\ln p}{p\ln n}-\frac{1}{2},\end{gather*}
therefore it suffices to prove that $\frac{n\ln p}{p\ln n}\geq \frac{\sqrt{n}}{2}$. We observe that the function 
\begin{gather*}f(x)=\frac{\ln x}{x}\end{gather*}
is strictly decreasing for $x\geq e$, since $f'(x)=\frac{1-\ln x}{x^2}$. But $e<3\leq p\leq \sqrt{n}$, therefore 
\begin{gather*}f(p)\geq f(\sqrt{n})\Leftrightarrow \frac{\ln p}{p}\geq \frac{\ln n}{2\sqrt{n}}.\end{gather*}
Finally, we multiply both sides by $\frac{n}{\ln n}$, to obtain 
\begin{gather*}\frac{n\ln p}{p\ln n}\geq \frac{n}{\ln n}\cdot\frac{\ln n}{2\sqrt{n}} =\frac{\sqrt{n}}{2}.\end{gather*}
This completes the proof.
\end{proof}
\end{lemma}
\section{Main results}
In this section we state our main results and provide some key proof techniques and insights. We begin with the following:
\begin{theorem}\label{main}
Suppose that $n\geq (2c+3)^2$ and $n$ is not of the form $(c+1)p$, where $p$ is prime. Then 
\begin{equation}\label{eq5}\sum_{M\in \mathcal{N}[n-1-c]}M\equiv 0\pmod{n}.\end{equation}
\begin{proof}
The main idea of the proof is to show that for every prime power $p^a$ with $p^a\mid \mid n$, then $p^a\mid M$, for every $M\in \mathcal{N}[n-1-c]$. 
With the aim of doing this, suppose that $p$ is a prime divisor of $n$, where 
\begin{gather*}n=kp^a,\quad p^a\mid \mid n \quad \text{and}\quad p^x\leq n<p^{x+1}.\footnote{In particular we will show that $p^x\mid M$, which is even stronger.}\end{gather*}

It is important to note that $n=kp^a$ and $n\neq (c+1)p$, imply that either $a\geq 2$ or $k\neq c+1$; thus, it suffices to consider the following cases:
 \begin{case 1*}\emph{$a\geq 2$.}\end{case 1*}
Recall that \eqref{eq2} states
\begin{gather*}\sum_{M\in \mathcal{N}[n-1-c]}M=\sum_{M\in \mathcal{N}[c]}\frac{(n-1)!}{M}.\end{gather*}
Every $M\in \mathcal{N}[c]$ is a product of $c$ positive integers less than $n$. Thus, the largest power of $p$ that divides $M$ is at most $p^{cx}$. Since $\nu_p\left((n-1)!\right)$ is the exponent of the largest power of $p$ that divides $(n-1)!$, it will be enough to show that 
\begin{equation}\label{eq6}\nu_p\left((n-1)!\right)-cx\geq a.\end{equation}
Let $p\geq 3$. The assumption $a\geq 2$ implies $p\leq \sqrt{n}$, hence, applying Lemma \eqref{lemma2} we get 
\begin{gather*}\left(\frac{n}{p}-1\right)\frac{\ln p}{\ln n}\geq \frac{\sqrt{n}-1}{2}\geq c+1.\end{gather*}
The last part of the inequality comes from the fact that $n\geq (2c+3)^2$. It follows that
\begin{equation}\label{eq7}\frac{n}{p}-1\geq (c+1)\frac{\ln n}{\ln p}\geq (c+1)x,\end{equation}
because by assumption, $p^x\leq n$ which implies $x\leq\frac{\ln{n}}{\ln{p}}$. Legendre's formula yields 
\begin{gather*}\nu_p\left((n-1)!\right)\geq \left\lfloor\frac{n-1}{p}\right\rfloor=\frac{n}{p}-1,\end{gather*}
since $p\mid n$. From \eqref{eq7} and the previous inequality it follows that $\nu_p\left((n-1)!\right)\geq (c+1)x$, which implies that $\nu_p\left((n-1)!\right)-cx\geq x\geq a$. Therefore, \eqref{eq6} is proved.\par
If $p=2$, we can prove \eqref{eq6} fairly easily (even if $a<2$): From Legendre's formula and property \eqref{item3}, 
\begin{equation}\label{eq8}\nu_2\left((n-1)!\right)=\left\lfloor\frac{n-1}{2}\right\rfloor+\cdots+\left\lfloor\frac{n-1}{2^x}\right\rfloor\geq\frac{n}{2}-1+\cdots +\frac{n}{2^x}-1=n\cdot\frac{2^x-1}{2^x}-x.\end{equation}
It is easy to check that $n\geq(2c+3)^2$ and $c\geq 1$ imply that $n\geq 25$. By assumption, $2^x\leq n<2^{x+1}$, thus $x\geq 4$ and consequently, $\frac{2^x-1}{2^x}\geq \frac{15}{16}$. Therefore, \eqref{eq8} yields
\begin{equation}\label{eq9}\nu_2\left((n-1)!\right)\geq \frac{15n}{16}-x.\end{equation} 
Moreover, $n\geq(2c+3)^2$ and $2^x\leq n$ are equivalent to $c+2\leq \frac{\sqrt{n}+1}{2}$ and $x\leq \frac{\ln x}{\ln 2}$, respectively. Thus, from Lemma \eqref{theorem1} we obtain 
\begin{gather*}\frac{15n}{16}>\frac{(\sqrt{n}+1)\ln n}{2\ln 2}\geq (c+2)\frac{\ln n}{\ln 2}\geq (c+2)x,\end{gather*}
which implies that $\frac{15n}{16}-x-cx\geq x$. From \eqref{eq9}, it follows that $\nu_2\left((n-1)!\right)-cx\geq x\geq a$, which completes the proof of the first case.
\begin{case 2*}\emph{$a=1, k\neq c+1$.}\end{case 2*}
Recall that from \eqref{eq3},
\begin{gather*}\sum_{M\in \mathcal{N}[n-1-c]}M\equiv\sum_{M\in \mathcal{N}_1[c]}\frac{(n-1)!}{M}\pmod{p}.\end{gather*}
Since $a=1$, then $n=kp$. Thus, there are $\left\lfloor\frac{n-1}{p}\right\rfloor=k-1$ multiples of $p$ less than $n$.\par
If $k>c+1$, then $k-1>c$. Thus, by definition, $\mathcal{N}_1[c]=\emptyset$. Hence,
\begin{gather*}\sum_{M\in \mathcal{N}[n-1-c]}M\equiv0\pmod{p}.\end{gather*}
If $k<c+1$, then $k-1<c$. For every $M\in \mathcal{N}_1[c]$, $\frac{(n-1)!}{M}$ is not divisible by $p$, since the $k-1$ multiples of $p$ are canceled from the numerator and the denominator. In particular, if
\begin{gather*}(n-1)!_p=\prod_{\substack{p\nmid m\\m<n}}m,\end{gather*}
then
\begin{equation}\label{eq10}\sum_{M\in \mathcal{N}_1[c]}\frac{(n-1)!}{M}=\sum_{\substack{M\in \mathcal{N}[c-(k-1)]\\p\nmid M}}\frac{(n-1)!_p}{M}.\end{equation} 
A simple counting argument shows that there are $k(p-1)$ integers less than $n$, not divisible by $p$. Thus, the right-hand side of \eqref{eq10}, is equal to the sum of the products of these numbers, taken $k(p-1)-(c-(k-1))$ at a time, which is the coefficient of $x^{c-(k-1)}$ in the polynomial 
\begin{gather*}P(x)=\prod_{\substack{m\leq n-1\\p \nmid m}}(x-m)=\prod_{j=0}^{k-1}\left(x-(jp+1)\right)\cdots\left(x-(jp+p-1)\right).\end{gather*}
But 
\begin{gather*}\left(x-(jp+1)\right)\cdots\left(x-(jp+p-1)\right)\equiv (x-1)\cdots (\left(x-(p-1)\right)\pmod{p},\end{gather*}
thus, \eqref{eq2}, yields
\begin{gather*}P(x)\equiv(x^{p-1}-1)^k=\sum_{i=0}^{k}\binom{k}{i}(-1)^ix^{(p-1)(k-i)}\pmod{p}.\end{gather*}
Therefore, in order to show that $\displaystyle \sum_{M\in \mathcal{N}_1[c]}\frac{(n-1)!}{M}\equiv 0\pmod{p}$, it suffices to show that the coefficient of $x^{c-(k-1)}$ vanishes. Since the only coefficients which are (possibly) non-zero modulo $p$ correspond to the power $x^{(p-1)(k-i)}$, it suffices to show that $c-(k-1)\neq (p-1)(k-i)$.\par
Aiming for a contradiction, we suppose that $c-(k-1)=(k-i)(p-1)$. This implies that $p-1\mid c-k+1$, hence, $p-1\leq c-k+1$. But $k\geq 1$, thus $p\leq c+1$. We combine this with the hypothesis that $k<c+1$ and $n\geq (2c+3)^2$, to conclude 
\begin{gather*}(2c+3)^2\leq n=kp<(c+1)^2\end{gather*}
which is absurd since $c\geq 1$. The proof is complete.
\end{proof}
\end{theorem}
\begin{theorem} Let $n=(c+1)p$. Then
\begin{gather*}\left((-1)^c+\sum_{M\in\mathcal{N}[n-1-c]}M\right)\cdot(c+1)\equiv 0\pmod{n}.\end{gather*}
\begin{proof}
It suffices to prove that
\begin{gather*}(-1)^c+\sum_{M\in\mathcal{N}[n-1-c]}M\equiv 0\pmod{p}.\end{gather*}
If $n=(c+1)p$, there are exactly $c$ multiples of $p$ less than $n$, which implies that $\mathcal{N}_1[c]$ contains only one element. In particular, $\mathcal{N}_1[c]=\{p\cdot 2p\cdots cp\}$ and
\begin{gather*}\sum_{M\in \mathcal{N}_1[c]}\frac{(n-1)!}{M}=(n-1)!_p.\end{gather*}
Using Wilson's theorem we obtain
\begin{gather*}(n-1)!_p=\prod_{j=0}^{c}(jp+1)\cdots(jp+p-1)\equiv \left((p-1)!\right)^{c+1}\equiv (-1)^{c+1}\pmod{p}.\end{gather*}
Thus,
\begin{gather*}\sum_{M\in \mathcal{N}[n-1-c]}M\equiv(-1)^{c+1}\pmod{p}.\end{gather*}
Consequently,
\begin{gather*}\left((-1)^c+\sum_{M\in \mathcal{N}[n-1-c]}M\right)\cdot(c+1)\equiv 0\pmod{n},\end{gather*}
which completes the proof.
\end{proof}
\end{theorem}
In this way, we obtain the following corollary:
\begin{corollary} If $n\geq (2c+3)^2$, then 
\begin{gather*}\left((-1)^c+\sum_{M\in \mathcal{N}[n-1-c]}M\right)\cdot(c+1)\equiv 0\pmod{n},\end{gather*}
if and only if $n=(c+1)p$, where $p$ is prime.
\begin{proof}
If $c=0$, the claim is equivalent to Wilson's Theorem. Therefore, we assume that $c\geq 1$ and it suffices to show that the above congruence does not hold if $n\neq (c+1)p$. We observe that if $n\neq (c+1)p$, Theorem \eqref{main} implies that
\begin{gather*}\left((-1)^c+\sum_{M\in \mathcal{N}[n-1-c]}M\right)\cdot(c+1)\equiv (-1)^c\cdot (c+1)\pmod{p^a},\end{gather*}
for every $p^a\mid\mid n$. This is equivalent to $(-1)^c\cdot(c+1)\equiv 0\pmod{n}$ and consequently, $n\leq c+1$. This is absurd, since $n\geq (2c+3)^2$. This completes the proof and the paper.
\end{proof}
\end{corollary}

\makeatletter
\renewcommand{\@biblabel}[1]{[#1]\hfill}
\makeatother

\end{document}